\documentclass[a4j,12pt,final]{article}
\usepackage{amsmath}
\usepackage{amssymb}
\usepackage{amsthm}
\usepackage{amscd}
\newtheorem{theorem}{Theorem}[section]

\theoremstyle{definition}
\newtheorem{definition}{Definition}[section]
\theoremstyle{remark}
\newtheorem{remark}{Remark}[section]
\theoremstyle{proposition}

\numberwithin{equation}{section}
\newtheorem{example}{Example}[section]
\theoremstyle{corollary}
\newtheorem{corollary}{Corollary}[section]

\theoremstyle{conjecture}

\usepackage{fancyhdr}
\begin{document}

\title{On some Chern-Simons forms of the Bott-Shulman-Stasheff forms}
\author{Naoya Suzuki}
\date{}
\maketitle
\begin{abstract}
We exhibit the Chern-Simos forms of some characteristic classes in the simplicial de Rham complex.
\end{abstract}


\section{Introduction}
\setcounter{equation}{0}
%

In the framework of differential geometry on the simplicial manifold,  the 
author exhibited some cocycles  in ${\Omega}^{*} (NG(*)) $ which represent classical characteristic classes of the universal bundle $EG \rightarrow BG$\cite{Suz}\cite{Suz2} on the basis of Dupont's work\cite{Dup}. Here $NG$ is a simplicial manifold called nerve of $G$ and it is well-known that the cohomology ring of
${\Omega}^{*} (NG(*)) $ is isomorphic to $H^*(BG)$ for any Lie group $G$. These cocycles in ${\Omega}^{*} (NG(*)) $ are called the Bott-Shulman-Stasheff forms.

On the other hand, there is a simplicial manifold $PG$ which play the role of $EG$.
Since $H^*(EG)$ is trivial, any cocycle in $\Omega^*(PG)$ is exact. So if we pullback the Bott-Shulman-Stasheff form
to  $\Omega^*(PG)$, there exists a cochain  $\Omega^{*-1}(PG)$ which hits the cocycle by a coboundary operator.
These forms can be called the Chen-Simons forms of the BSS forms.

In this paper, we exhibit the Chen-Simons forms of the BSS forms which represent some classical characteristic classes.

\section{Review of the BSS forms}

In this section we recall the universal Chern-Weil theory following \cite{Dup2}.
For any Lie group $G$, we define simplicial manifolds $NG$, $PG$ and simplicial $G$-bundle  $\gamma : PG \rightarrow NG$
as follows:\\
\par
$NG(q)  = \overbrace{G \times \cdots \times G }^{q-times}  \ni (h_1 , \cdots , h_q ) :$  \\
face operators \enspace ${\varepsilon}_{i} : NG(q) \rightarrow NG(q-1)  $
$$
{\varepsilon}_{i}(h_1 , \cdots , h_q )=\begin{cases}
(h_2 , \cdots , h_q )  &  i=0 \\
(h_1 , \cdots ,h_i h_{i+1} , \cdots , h_q )  &  i=1 , \cdots , q-1 \\
(h_1 , \cdots , h_{q-1} )  &  i=q.
\end{cases}
$$

\par
\medskip
$PG (q) = \overbrace{ G \times \cdots \times G }^{q+1 - times} \ni (g_0 , \cdots , g_{q} ) :$ \\
face operators \enspace $ \bar{\varepsilon}_{i} :PG(q) \rightarrow PG(q-1)  $ 
$$ \bar{{\varepsilon}} _{i} (g_0 , \cdots , g_{q} ) = (g_0 , \cdots , g_{i-1} , g_{i+1}, \cdots , g_{q})  \qquad i=0,1, \cdots ,q. $$

\par
\medskip

We define $\gamma : PG \rightarrow NG $ as $ \gamma (g_0 , \cdots , g_{q} ) = (g_0 {g_1}^{-1} , \cdots , g_{q-1} {g_{q}}^{-1} )$ then $\parallel \gamma \parallel$ is a model of the universal bundle $EG \rightarrow BG$  \cite{Seg}. 

There is a double complex associated to a simplicial manifold.

\begin{definition}
For any simplicial manifold $ \{ X_* \}$ with face operators $\{ {\varepsilon}_* \}$, we define a double complex as follows:
$${\Omega}^{p,q} (X) := {\Omega}^{q} (X_p). $$
Derivatives are:
$$ d' := \sum _{i=0} ^{p+1} (-1)^{i} {\varepsilon}_{i} ^{*}  , \qquad  d'' := (-1)^{p} \times {\rm the \enspace exterior \enspace differential \enspace on \enspace }{ \Omega ^*(X_p) }. $$
\end{definition}
\bigskip

For $NG$ and $PG$ the following holds \cite{Bot2} \cite{Dup2} \cite{Mos}.

\begin{theorem}
There exist ring isomorphisms 
$$ H({\Omega}^{*} (NG))  \cong  H^{*} (BG ), \qquad  H({\Omega}^{*} (PG)) \cong H^{*} (EG ). $$
Here ${\Omega}^{*} (NG)$  and  ${\Omega}^{*} (PG)$  mean the total complexes.
\end{theorem}

There is another double complex associated to a simplicial manifold.
\begin{definition}[\cite{Dup}]
A simplicial $n$-form on a simplicial manifold $ \{ {X}_{p} \} $ is a sequence $ \{ {\phi}^{(p)} \}$
of $n$-forms ${\phi}^{(p)}$ on ${\Delta}^{p} \times {X}_{p} $ such that
$${({\varepsilon}^{i} \times id )}^{*} {\phi}^{(p)} = {(id \times {\varepsilon}_{i} )}^{*} {\phi}^{(p-1)}. $$
Here ${\varepsilon}^{i}$ is the canonical $i$-th face operator of ${\Delta}^{p}$.\\
\end{definition}

{\rm Let}  \thinspace $A^{k,l} (X)$ be the set of all simplicial $(k+l)$-forms on ${\Delta}^{p} \times {X}_{p} $ which are expressed locally 
of the form
$$ \sum { a_{ i_1 \cdots i_k j_1 \cdots j_l } (dt_{i_1 } \wedge \cdots \wedge dt_{i_k } \wedge dx_{j_1 } \wedge \cdots \wedge dx_{j_l })}$$
where $(t_0, t_1, \cdots, t_p)$ are the barycentric coordinates in ${\Delta}^{p}  $ and $x_j $ are the local coordinates in $ {X}_{p} $.
We define its derivatives as:
$$ d' := {\rm  the \enspace exterior \enspace differential \enspace on \enspace } {\Delta}^{p},  $$
$$ d'' := (-1)^{k} \times {\rm  the \enspace exterior \enspace differential \enspace on \enspace } {X_p }. $$
Then $(A^{k,l} (X) , d' , d'' )$ is a double complex.\\

\begin{theorem}[\cite{Dup}]
Let $A^{*} (X)$ denote the total complex of $A^{*,*}(X)$. A map $ I_{ \Delta }( \alpha ) :=  \int_{{\Delta }^{p}} ( { \alpha } |_{{ \Delta }^{p} \times {X}_{p} } )$
{ induces a natural ring isomorphism} $ I_{ \Delta } ^{*} : H( A^{*} (X)) \cong H({\Omega}^{*} (X))$.

\end{theorem}
\bigskip

Let  $\mathcal{G}$ denote the Lie algebra of $G$. A connection on a simplicial $G$-bundle $\pi : \{ E_p \} \rightarrow \{ M_p \} $ is a sequence of $1$-forms $\{ \theta \}$ on $\{ E_p \}$ with coefficients $\mathcal{G}$
such that $\theta $ restricted  to ${\Delta}^{p} \times {E}_{p} $ is a usual connection form on a principal $G$-bundle  ${\Delta}^{p} \times {E}_{p} \rightarrow {\Delta}^{p} \times {M}_{p} $.
There is a canonical connection $\theta \in A^1 (PG)$ on ${\gamma} : PG \rightarrow NG $ defined as $ {\theta } |_{{ \Delta }^{p} \times N \bar{G} (p)} := t_0 {\theta }_0 + \cdots + t_{p} {\theta }_{p}$.
Here ${\theta }_i $ is defined by ${\theta }_i = {\rm pr}_i ^{*} \bar {\theta } $ where ${\rm pr}_i : { \Delta }^{p} \times PG (p) \rightarrow G $ is the projection into the $i$-th factor of $ PG (p) $ and $\bar {\theta }$ is the Maurer-Cartan form of $G$.
We also obtain its curvature $\Omega \in A^2 (PG)$ on ${\gamma }$ as: $ \Omega |_{{ \Delta }^{p} \times PG (p) }= d \theta |_{{ \Delta }^{p} \times PG (p) }
+  \frac{1}{2} [ \theta |_{{ \Delta }^{p} \times PG (p) } , \theta |_{{ \Delta }^{p} \times PG (p) } ]$.

Let  ${\rm  I}^{*} (G)$ denote the ring of $G$-invariant polynomials on $\mathcal{G}$. For $P \in I^* (G)$, we restrict 
$P( \Omega ) \in A^{*} (PG)$ to each ${\Delta}^{p} \times PG (p) 
\rightarrow {\Delta}^{p} \times NG(p) $ and apply the usual Chern-Weil theory then we have a simplicial $2k$-form  $P( \Omega )$
on $NG$.

 Now we have the universal Chern-Weil homomorphism ${w}:{\rm  I}^{*} (G) \rightarrow H({\Omega}^{*} (NG)) $
which maps $P \in I^* (G)$  to ${w}(P)=[I_{\Delta } ( P({\Omega}) )]$. The images of this homomorphism in ${\Omega}^{*} (NG(*)) $ are called the Bott-Shulman-Stasheff(BSS) forms.

\begin{example}
In the case that $G=U(n)$, the BSS form which represents the $p$-th power of the first Chern class $c_1^p~(={\rm ch}_1^p) \in H^{2p}(BU(n))$ is given as follows:
$$ \left( \frac{1}{2 \pi i } \right)^p (-1)^{\frac{p(p-1)}{2}}{\rm tr }(h_1^{-1} dh_1 ){\rm tr }(h_2^{-1} dh_2 )\cdots{\rm tr }(h_p^{-1} dh_p ) \in \Omega^p(NG(p)). $$

\end{example}

\begin{remark}
The cocycle in the example 2.1 is given as the image of $\displaystyle ({\rm tr}\left( \frac{-X}{2\pi i}  \right))^p $ under the universal Chern-Weil homomorphism. 
On the other hand, there is a product $\cup$ on ${\Omega}^{*} (NG)$ defined as follows:
$$\alpha \cup \beta = (-1)^{qr} {\varepsilon}_{s+t}^* \cdots {\varepsilon}_{s+1}^* \alpha \wedge {\varepsilon}_{0}^* \cdots {\varepsilon}_{0}^* \beta,~~~~(\alpha \in {\Omega}^{q} (NG(s)), \beta \in {\Omega}^{r} (NG(t))).$$
We can see $c_1 \cup \cdots \cup c_1$ coincides with the cocycle in this example as a cochain.
\end{remark}

\begin{example}
In the case that $G=U(n)$, the BSS form which represents the $3$rd  Chern class ${\rm c}_3~(=2{\rm ch}_3-{\rm ch}_2{\rm ch}_1+\frac{1}{6}
{\rm ch}_1^3) \in H^6(BU(n))$ in $ \Omega ^{6} (NG) $ is the sum of the following
$c_{1,5} , c_{2,4}$ and $c_{3,3}$:
$$
\begin{CD}
0 \\
@AA{d''}A \\
{ c_{1,5} \in {\Omega}^{5} (G )}@>{d'}>>{\Omega}^{5} (NG(2))\\
@.@AA{d''}A\\
@.{ c_{2,4} \in {\Omega}^{4} (NG(2))}@>{d'}>> {\Omega}^{4} (NG(3)) \\
@.@.@AA{d''}A\\
@.@.{ c_{3,3} \in {\Omega}^{3} (NG(3))}@>{d'}>> 0
\end{CD}
$$

$$c_{1,5} = \frac{1}{6} \left( \frac{1}{2 \pi i} \right) ^3 \frac{1 }{5}{\rm tr}({h^{-1}dh} )^5,  \hspace{18em} $$

$$    c_{2,4}= \frac{-1}{6} \left( \frac{1}{2 \pi i} \right) ^3 (
{\rm tr}  (   dh_1 {h_1}^{-1}dh_1 {h_1}^{-1}dh_1 dh_2 {h_2}^{-1}{h_1}^{-1})  \hspace{7em}$$
 $$ + \frac{1 }{2}  {\rm tr}  ( dh_1 dh_2 {h_2}^{-1}{h_1}^{-1}dh_1 dh_2 {h_2}^{-1}{h_1}^{-1}) $$
$$\hspace{9em} +  {\rm tr}   ( dh_1 dh_2 {h_2}^{-1}dh_2 {h_2}^{-1}dh_2 {h_2}^{-1}{h_1}^{-1}))  $$
$$-\frac{1}{12} \left( \frac{1}{2 \pi i} \right) ^3({\rm tr} ({h_2}^{-1}dh_2 {h_2}^{-1}dh_2{h_2}^{-1}dh_2+          dh_1dh_2h_2^{-1}dh_2h_2^{-1}h_1^{-1}$$
$$  + dh_1h_1^{-1}dh_1dh_2h_2^{-1}h_1^{-1}){\rm tr} (h_1^{-1}dh_1)$$
$$-{\rm tr}(h_1^{-1}dh_1h_1^{-1}dh_1h_1^{-1}dh_1+dh_1h_1^{-1}dh_1dh_2h_2^{-1}h_1^{-1}$$
$$+dh_1dh_2h_2^{-1}dh_2h_2^{-1}h_1^{-1}) {\rm tr} (h_2^{-1}dh_2)),$$

$$  c_{3,3} =\frac{-1}{6} \left( \frac{1}{2 \pi i} \right) ^3 ({\rm tr} ( dh_1 dh_2 dh_3 {h_3}^{-1}{h_2}^{-1}{h_1}^{-1} ) \hspace{10em}$$
$$  \hspace{10em} - {\rm tr} (dh_1 h_2 dh_3 {h_3}^{-1}{h_2}^{-1}  dh_2 {h_2}^{-1}{h_1}^{-1}))$$
$$+\frac{1}{6} \left( \frac{1}{2 \pi i} \right) ^3({\rm tr} (dh_1dh_2h_2^{-1}h_1^{-1}){\rm tr} (h_3^{-1}dh_3)
+{\rm tr}(dh_2dh_3h_3^{-1}h_2^{-1}){\rm tr} (h_1^{-1}dh_1)$$
$$-{\rm tr}(dh_1h_2dh_3h_3^{-1}h_2^{-1}h_1^{-1}){\rm tr} (h_2^{-1}dh_2))$$

$$-\frac{1}{6} \left( \frac{1}{2 \pi i} \right) ^3{\rm tr} (h_1^{-1}dh_1){\rm tr} (h_2^{-1}dh_2){\rm tr} (h_3^{-1}dh_3). $$

\end{example}

\begin{remark}
The cocycle in the example 2.2 is given as the image of 
$ \frac{1}{3}{\rm tr}\left( \left( \frac{-X}{2\pi i} \right)^3 \right) -\frac{1}{2}{\rm tr}\left( \left( \frac{-X}{2\pi i} \right)^2 \right){\rm tr}\left(  \frac{-X}{2\pi i} \right)+\frac{1}{6}
({\rm tr}\left(  \frac{-X}{2\pi i} \right))^3$ under the universal Chern-Weil homomorphism.
Unfortunately, the cocycle $2{\rm ch}_3-{\rm ch}_2 \cup {\rm ch}_1+\frac{1}{6}
{\rm ch}_1 \cup {\rm ch}_1 \cup {\rm ch}_1 \in {\Omega}^{6} (NG)$ does not coincide with the one in this example as a cochain.

\end{remark}
\bigskip


\section{The Chern-Simons forms of the BSS forms}

Since $H^*(EG)$ is trivial, any cocycle in $\Omega^*(PG)$ is exact. So if we pullback the Bott-Shulman-Stasheff form
to  $\Omega^*(PG)$, there exists a cochain  $\Omega^{*-1}(PG)$ which hits the cocycle by a coboundary operator.
These forms can be called the Chen-Simons forms of the BSS forms.

In this section, we exhibit the Chen-Simons forms of the BSS forms which represent some classical characteristic classes.

\subsection{The Chern-Simons form of the universal torus bundle}


In this subsection we exhibit the Chern-Simons form of the Chern characters and Chern classes of the universal torus bundle.

We write the $s$-th factor of $PT^n(p)$ as: 
$$g_s=
\begin{pmatrix}
{\rm exp}(i\theta^s_1) & \dots &0\\
& \ddots &\\
0 & \dots &{\rm exp}(i\theta^s_n)

\end{pmatrix}~\in T^n,~~(s=0, \cdots,p).$$

\begin{theorem}[\cite{Suz3}]
The cocycle $\bar{\omega}_{p}$ in $ \Omega ^{p} (PT^n(p)) $ which corresponds to the $p$-th Chern character of the universal torus bundle is given as follows:
$$\bar{\omega}_{p}= (-1)^{\frac{p(p-1)}{2} } \frac{1}{p! } \left( \frac{1}{2 \pi  } \right)^p   \sum _{k=1} ^n \left( \prod_{s=0} ^{p-1} \left(d \theta^s_k - d \theta^{s+1}_k \right) \right).$$
\end{theorem}

\begin{theorem}
The Chern-Simons form $T\bar{\omega}_{p-1}$ of the $p$-th Chern character  of the universal torus bundle in $ \Omega ^{p} (PT^n(p-1)) $ is given as follows:
$$T\bar{\omega}_{p-1}= (-1)^{\frac{p(p+1)}{2}   } \frac{1}{p! } \left( \frac{1}{2 \pi } \right)^p   \sum _{k=1} ^n \left(d \theta ^0_k \wedge
d \theta ^1_k \wedge \cdots \wedge d \theta ^{p-1}_k \right).$$
\end{theorem}

\begin{proof}
$$ \prod_{s=0} ^{p-1} \left(d \theta^s_k - d \theta^{s+1}_k \right)=\sum^{p}_{j=0} (-1)^{p-j} d\theta^0_k  \wedge \cdots \wedge d\theta_k^{j-1} \wedge d\theta_k^{j+1} \wedge \cdots \wedge d\theta_k^p $$
$$=(-1)^p \sum^{p}_{j=0}\bar{{\varepsilon}}^*_j(d \theta ^0_k \wedge
d \theta ^1_k \wedge \cdots \wedge d \theta ^{p-1}_k ).$$
So we can see that $(d'+d'')T\bar{\omega}_{p-1}=\bar{\omega}_{p}$.
\end{proof}

\begin{theorem}[\cite{Suz3}]
The cocycle $\bar{\mu}_{p}$ in $ \Omega ^{p} (PT^n(p)) $ which corresponds to the $p$-th Chern class of the universal torus bundle is given as follows:
$$\bar{\mu}_{p}= (-1)^{\frac{p(p-1)}{2}   } \frac{1}{p! } \left( \frac{1}{2 \pi  } \right)^p   \sum _{1 \le k_1 < \cdots < k_p \le n} {\rm det} \left( d \theta^s_{k_t} - d \theta^{s+1}_{k_t} \right)_{0 \le s,t \le p-1}.$$
\end{theorem}

\begin{theorem}
The Chern-Simons form $T\bar{\mu}_{p-1}$ of the $p$-th Chern class of the universal torus bundle  in $ \Omega ^{p} (PT^n(p-1)) $ is given as follows:
$$T\bar{\mu}_{p-1}= (-1)^{\frac{p(p+1)}{2}   } \frac{1}{p! } \left( \frac{1}{2 \pi  } \right)^p   \sum _{1 \le k_1 < \cdots < k_p \le n} {\rm det} \left( d \theta^s_{k_t}  \right)_{0 \le s,t \le p-1}.$$
\end{theorem}

\begin{proof}
We write:
$${\bf d \theta^j}=
\begin{pmatrix}
d \theta^j_{k_0}\\
\vdots \\
d \theta^j_{k_{p-1}}

\end{pmatrix},~~~(j=0, \cdots,p-1),$$

$$\left( d \theta^s_{k_t}  \right)_{0 \le s,t \le p-1}=({\bf d \theta^0}~ {\bf d \theta^1} \cdots {\bf d \theta^{p-1}}).$$
Then the following equations hold.
$$\bar{{\varepsilon}}^*_0 {\rm det} ({\bf d \theta^0}~ {\bf d \theta^1} \cdots {\bf d \theta^{p-1}})-\bar{{\varepsilon}}^*_1 {\rm det} ({\bf d \theta^0}~ {\bf d \theta^1} \cdots {\bf d \theta^{p-1}})$$
$$={\rm det} ({\bf d \theta^1}~ {\bf d \theta^2} \cdots {\bf d \theta^{p}})-{\rm det} ({\bf d \theta^0}~ {\bf d \theta^2} \cdots {\bf d \theta^{p}})$$
$$={\rm det} (({\bf d \theta^1}-{\bf d \theta^0})~ {\bf d \theta^2} \cdots {\bf d \theta^{p}}),$$
$${\rm det} (({\bf d \theta^1}-{\bf d \theta^0})~ {\bf d \theta^2} \cdots {\bf d \theta^{p}})+\bar{{\varepsilon}}^*_2 {\rm det} ({\bf d \theta^0}~ {\bf d \theta^1} \cdots {\bf d \theta^{p-1}})$$
$$={\rm det} (({\bf d \theta^1}-{\bf d \theta^0})~ {\bf d \theta^2}~ {\bf d \theta^3} \cdots {\bf d \theta^{p}})+ {\rm det} ({\bf d \theta^0}~ {\bf d \theta^1}~ {\bf d \theta^3} \cdots {\bf d \theta^{p-1}})$$
$$={\rm det} (({\bf d \theta^1}-{\bf d \theta^0})~ {\bf d \theta^2}~ {\bf d \theta^3} \cdots {\bf d \theta^{p}})+ {\rm det} (({\bf d \theta^1}-{\bf d \theta^0})~ (-{\bf d \theta^1})~ {\bf d \theta^3} \cdots {\bf d \theta^{p-1}})$$
$$={\rm det} (({\bf d \theta^1}-{\bf d \theta^0})~ ({\bf d \theta^2}-{\bf d \theta^1})~ {\bf d \theta^3} \cdots {\bf d \theta^{p}}).$$
Repeating this argument, we can see that $(d'+d'')T\bar{\mu}_{p-1}=\bar{\mu}_{p}$.
\end{proof}

\begin{remark}
In the case that $n$ is equal to $p$, the cochain in theorem 3.4 is written as follows:
$$T\bar{\mu}_{p-1}= (-1)^{\frac{p(p+1)}{2}  } \frac{1}{p! } \left( \frac{1}{2 \pi } \right)^p 
  {\rm det} \left( d \theta^s_{t} \right)_{0 \le s,t \le p-1}.$$

\end{remark}

\subsection{The Chern-Simons form of the $3$rd Chern character}


In this subsection we exhibit the Chern-Simons form of the $3$rd Chern character in $\Omega ^{5}(PG) $.  
Throughout this subsection, $G= GL(n ;  \mathbb{C} )$.

We first recall the cocycle in ${\Omega}^{p+q}(PG(p-q)) (0 \leq q \leq p-1)$ which corresponds to the $p$-th Chern character.

\begin{theorem}[\cite{Suz}]
{We set:}
$$ \bar{S}_{p-q}=\sum _{\sigma \in \mathfrak{S} _{p-q-1 } }({\rm sgn} ( \sigma ) ) (\theta _{ \sigma (1)}  - \theta _{ \sigma (1)+1 } )
\cdots (\theta _{ \sigma (p-q-1) }  - \theta _{ \sigma (p-q-1)+1 } ).$$
{Then the cocycle in } $ {\Omega}^{p+q} ( PG (p-q) ) \ (0 \leq q \leq p-1) $ { which corresponds to the $p$-th Chern character
${\rm ch}_p$ is}

$$ \frac{1}{p! } \left(\frac{1}{2 \pi i } \right)^p (-1)^{(p-q)(p-q-1)/2  } \times \hspace{18em} $$
$$    \mathrm{tr} \sum \left( (p(\theta _0 - \theta _{1})) \wedge \bar{H}_q (\bar{S}_{p-q}) \times  \int _{{\Delta}^{p-q}} \prod_{i<j} (t_i t_j)^{a_{ij}(\bar{H}_q (\bar{S}_{p-q}))} dt_1 \wedge \cdots \wedge dt_{p-q} \right).$$
Here $\bar{H}_q (\bar{S}_{p-q})$ means the terms that $ (\theta _{i} - \theta _{j} )^2 \enspace (0 \leq i < j \leq p-q) $ { are
put }$q${ -times between }$(\theta _{k-1} - \theta _{k} )$ { and }$(\theta _{l} - \theta _{l+1} )$ { in $ \bar{S}_{p-q}$ permitting overlaps;
 $a_{ij}(\bar{H}_q (\bar{S}_{p-q}))$
means the number of }$(\theta _{i} - \theta _{j} )^2 ${ in it.
$\sum$ means the sum of all such terms}.

\end{theorem}

As a corollary of this theorem, we obtain the cocycle which corresponds to the $3$rd  Chern character in $ \Omega ^{6} (PG) $.

\begin{corollary}
The cocycle which corresponds to the $3$rd  Chern character in $ \Omega ^{6} (PG) $ is the sum of the following
$\bar{C}_{1,5} , \bar{C}_{2,4}$ and $\bar{C}_{3,3}$:
$$
\begin{CD}
0 \\
@AA{d''}A \\
{ \bar{C}_{1,5} \in {\Omega}^{5} (PG(1))}@>{d'}>>{\Omega}^{5} (PG(2))\\
@.@AA{d''}A\\
@.{ \bar{C}_{2,4} \in {\Omega}^{4} (PG(2))}@>{d'}>> {\Omega}^{4} (PG(3)) \\
@.@.@AA{d''}A\\
@.@.{ \bar{C}_{3,3} \in {\Omega}^{3} (PG(3))}@>{d'}>> 0
\end{CD}
$$

$$\bar{C}_{1,5} = \frac{1}{3!} \left( \frac{1}{2 \pi i} \right) ^3 \frac{1 }{10}{\rm tr}(\theta_0-\theta_1)^5,   $$

$$    \bar{C}_{2,4}= \frac{-1}{3!} \left( \frac{1}{2 \pi i} \right) ^3 (
\frac{1 }{2}{\rm tr}  ( \theta_0-\theta_1)^3 ( \theta_1-\theta_2) ~~~~~~~~~~~~~~~~~~~~~~~~~ $$
 $$~~~~~~~~~ + \frac{1 }{4}  {\rm tr} ( \theta_0-\theta_1)( \theta_1-\theta_2) ( \theta_0-\theta_1) ( \theta_1-\theta_2)  $$
$$ ~~~~~~~~~~~~~~~~~~~~~~~~~~~~~~~~~~~~~+ \frac{1 }{2} {\rm tr}   ( \theta_0-\theta_1) ( \theta_1-\theta_2)^3 ),   $$

$$  \bar{C}_{3,3} =\frac{-1}{3!} \left( \frac{1}{2 \pi i} \right) ^3 (\frac{1 }{2}{\rm tr}( \theta_0-\theta_1) ( \theta_1-\theta_2)( \theta_2-\theta_3)~~~~~~~~~~~~~~$$
$$ ~~~~~~~~~~~~~~~~~~~~~~~~~~~~~~~~  -\frac{1 }{2} {\rm tr} ( \theta_0-\theta_1)( \theta_2-\theta_3)( \theta_1-\theta_2)). $$

\end{corollary}

\begin{theorem}
The Chern-Simons form of the $3$rd  Chern character in $ \Omega ^{5} (PG) $ is the sum of the following
$T\bar{C}_{0,5} , T\bar{C}_{1,4}$ and $T\bar{C}_{2,3}$:
$$
\begin{CD}
0 \\
@AA{d''=d}A \\
{T\bar{C}_{0,5} \in {\Omega}^{5} (G)}@>{d'}>>{ \bar{C}_{1,5}}\\
@.@AA{d''=-d}A\\
@. {T\bar{C}_{1,4} \in {\Omega}^{4} (PG(1))}@>{d'}>> { \bar{C}_{2,4} }\\
@.@.@AA{d''=d}A\\
@.@.{T\bar{C}_{2,3} \in {\Omega}^{3} (PG(2))}@>{d'}>> { \bar{C}_{3,3}}
\end{CD}
$$

$$T\bar{C}_{0,5} = \frac{-1}{3!} \left( \frac{1}{2 \pi i} \right) ^3 \frac{1 }{10}{\rm tr}(\theta_0^5),   $$

$$    T\bar{C}_{1,4}= \frac{-1}{3!} \left( \frac{1}{2 \pi i} \right) ^3 (
\frac{1 }{2}{\rm tr}  \theta_0^3  \theta_1  - \frac{1 }{4}  {\rm tr}  \theta_0 \theta_1 \theta_0\theta_1  + \frac{1 }{2} {\rm tr}    \theta_0 \theta_1^3 ),   $$

$$  T\bar{C}_{2,3} =\frac{1}{3!} \left( \frac{1}{2 \pi i} \right) ^3 (\frac{1 }{2}{\rm tr}( \theta_0 \theta_1\theta_2)-\frac{1 }{2} {\rm tr} ( \theta_0 \theta_2 \theta_1)). $$

\end{theorem}

\begin{proof}

$$d'{\rm tr}(\theta_0^5)={\rm tr}(\theta_1^5-\theta_0^5),~~~~~~~~~~~~~~~~~~~~~~~~~~~~~~~~~~~~~~~~~~~~~~~~~~~~~~~~~$$
$$d(
\frac{1 }{2}{\rm tr}  \theta_0^3  \theta_1  - \frac{1 }{4}  {\rm tr}  \theta_0 \theta_1 \theta_0\theta_1  + \frac{1 }{2} {\rm tr}    \theta_0 \theta_1^3 )~~~~~~~~~~~~~~~~~~~~~~~~~~~~~~~~~~~~~~~~~$$
$$=\frac{-1 }{2}{\rm tr} (\theta_0^4  \theta_1  -\theta_0^3  \theta_1^2 )+ \frac{1 }{2}  {\rm tr} ( \theta_0^2 \theta_1 \theta_0\theta_1 - \theta_0 \theta_1^2 \theta_0\theta_1)  - \frac{1 }{2} {\rm tr}  (  \theta_0^2 \theta_1^3 -\theta_0 \theta_1^4).$$
So we can see that $d'T\bar{C}_{0,5}+d'' T\bar{C}_{1,4}= \bar{C}_{1,5}$.
$$d'(
\frac{1 }{2}{\rm tr}  \theta_0^3  \theta_1  - \frac{1 }{4}  {\rm tr}  \theta_0 \theta_1 \theta_0\theta_1  + \frac{1 }{2} {\rm tr}    \theta_0 \theta_1^3 )~~~~~~~~~~~~~~~~~~~~~~~~~~~~~~~~~~~~~~~~~~~~~~~~~~$$
$$=\frac{1 }{2}{\rm tr}  \theta_1^3  \theta_2  - \frac{1 }{4}  {\rm tr}  \theta_1 \theta_2 \theta_1\theta_2  + \frac{1 }{2} {\rm tr}    \theta_1 \theta_2^3-(\frac{1 }{2}{\rm tr}  \theta_0^3  \theta_2  - \frac{1 }{4}  {\rm tr}  \theta_0 \theta_2 \theta_0\theta_2  + \frac{1 }{2} {\rm tr}    \theta_0 \theta_2^3)~~~~~~~~~~~~~~~~~~~~~~~~~~~~~~~~$$
$$+\frac{1 }{2}{\rm tr}  \theta_0^3  \theta_1  - \frac{1 }{4}  {\rm tr}  \theta_0 \theta_1 \theta_0\theta_1  + \frac{1 }{2} {\rm tr}    \theta_0 \theta_1^3,~~~~~~~~~~~~~~~~~~~~~~~~~~~~~~~~~~~~~~~~~~~~~~$$

\newpage
$$d(\frac{1 }{2}{\rm tr}( \theta_0 \theta_1\theta_2)-\frac{1 }{2} {\rm tr} ( \theta_0 \theta_2 \theta_1))~~~~~~~~~~~~~~~~~~~~~~~~~~~~~~~~~~~~~~~~~~~~~~~~~~~$$
$$=\frac{-1 }{2}{\rm tr}( \theta_0^2 \theta_1\theta_2-\theta_0 \theta_1^2\theta_2+\theta_0 \theta_1\theta_2^2)+\frac{1 }{2} {\rm tr} ( \theta_0^2 \theta_2 \theta_1-\theta_0 \theta_2^2 \theta_1+\theta_0 \theta_2 \theta_1^2).$$
So we can see that $d'T\bar{C}_{1,4}+d'' T\bar{C}_{2,3}= \bar{C}_{2,4}$. We can check also that $d''T\bar{C}_{0.5}=0$ and
$d' T\bar{C}_{2,3}= \bar{C}_{3,3}$.
\end{proof}

\subsection{The Chern-Simons form of the Euler class}


In this subsection we take $G=SO(4)$ and exhibit the Chern-Simons form of the Euler class in $\Omega ^{3}(PSO(4)) $.

For the Pfaffian ${\rm Pf} \in {\rm  I}^{2} (SO(4))$,
we put the canonical simplicial connection into $\displaystyle TPf( \theta )  = 2 \int_{0}^{1} {\rm Pf}(\theta \wedge (t \Omega + \frac{1}{2}t(t-1) [ \theta , \theta ]) )dt$ and integrate it along the standard simplex ${ \Delta }^{*}$, then we obtain the Chern-Simons form in $\Omega ^{3}(PSO(4)) $.

\begin{theorem}
The Chern-Simons form of the Euler class of $ESO(4) \rightarrow BSO(4)$ in $ \Omega ^{3} (PSO(4)) $ is the sum of the following
$T\bar{E}_{0,3}$ and $T\bar{E}_{1,2}$:
$$
\begin{CD}
0 \\
@AA{d''=-d}A \\
TE_{0,3} \in {\Omega}^{3} (SO(4))@>{d'}>>\bar{E}_{1,3}\\
@.@AA{d''=d}A\\
@.TE_{1,2} \in {\Omega}^{2} (SO(4) \times SO(4))@>{d'}>> \bar{E}_{2,2}
\end{CD}
$$
$$T\bar{E}_{0,3} =  \frac{-1}{96 \pi ^2} \sum_{\tau \in \mathfrak{S} _{4}}   {\rm sgn} (\tau)  \bigl((\theta_0)_{\tau (1) \tau(2)}(\theta_0^2) _{\tau (3) \tau(4)}\bigl),$$
$$ T\bar{E}_{1,2} =  \frac{-1}{64 \pi ^2} \sum_{\tau \in \mathfrak{S} _{4}}   {\rm sgn} (\tau) \bigl((\theta_0)_{\tau (1) \tau(2)}(\theta_1) _{\tau (3) \tau(4)}+(\theta_0) _{\tau (3) \tau(4)}(\theta_1) _{\tau (1) \tau(2)} \bigl).$$

\end{theorem}
\begin{proof}
Since $\Omega|_{{ \Delta }^{0} \times PG (0) }=0$ and $[ \theta_0 , \theta_0 ]=2\theta_0^2$,
$$T\bar{E}_{0,3}=\int_{0}^{1}t(t-1) dt\frac{1}{16 \pi ^2} \sum_{\tau \in \mathfrak{S} _{4}}   {\rm sgn} (\tau)  \bigl((\theta_0)_{\tau (1) \tau(2)}(\theta_0^2) _{\tau (3) \tau(4)} \bigl)$$
$$=  \frac{-1}{96 \pi ^2} \sum_{\tau \in \mathfrak{S} _{4}}   {\rm sgn} (\tau)  \bigl((\theta_0)_{\tau (1) \tau(2)}(\theta_0^2) _{\tau (3) \tau(4)}\bigl).$$
Since $\Omega|_{{ \Delta }^{1} \times PG (1) }=-dt_1 \wedge (\theta_0-\theta_1)-t_0t_1(\theta_0-\theta_1)^2$,
$$ T\bar{E}_{1,2} =\int_{0}^{1}t dt\frac{(-1)}{16 \pi ^2} \sum_{\tau \in \mathfrak{S} _{4}}{\rm sgn} (\tau) \bigl(
 \int_{0}^{1} (t_0\theta_0+t_1\theta_1)_{\tau (1) \tau(2)}\wedge dt_1 \wedge (\theta_0-\theta_1)_{\tau (3) \tau(4)}  \bigl)$$
$$=  \frac{-1}{64 \pi ^2} \sum_{\tau \in \mathfrak{S} _{4}}   {\rm sgn} (\tau) \bigl((\theta_0)_{\tau (1) \tau(2)}(\theta_1) _{\tau (3) \tau(4)}+(\theta_0) _{\tau (3) \tau(4)}(\theta_1) _{\tau (1) \tau(2)} \bigl).$$
\end{proof}

National Institute of the Technology, Akita College, 1-1, Iijima Bunkyo-cho, Akita-shi, Akita-ken, Japan. \\
e-mail: nysuzuki@akita-nct.ac.jp
\end{document}